\documentclass[final,leqno]{siamltex}

\usepackage[english]{babel}
\usepackage{graphicx,epstopdf}
\usepackage{latexsym,bm}
\usepackage{newtxmath}   

\newtheorem{remark}{Remark}

\def\L{{\mathcal L}}
\def\i{{\rm i}}

\def\RR{{\mathbb R}}
\def\SS{{\mathbb S}}

\author{Evgeny Y. Derevtsov\thanks{Sobolev Institute of Mathematics SB RAS, 4 Acad. Koptyug avenue, 630090 Novosibirsk, Russia ({\tt dert@math.nsc.ru})} \and Thomas Schuster\thanks{Department of Mathematics, Saarland University, PO Box 15 11 50, 66123 Saarbr\"ucken, Germany ({\tt thomas.schuster@num.uni-sb.de})} \and Yuriy S. Volkov\thanks{Sobolev Institute of Mathematics SB RAS, 4 Acad. Koptyug avenue, 630090 Novosibirsk, Russia ({\tt volkov@math.nsc.ru})}}

\title{Generalized attenuated ray transforms and their integral angular moments}

\begin{document}

\maketitle


\begin{abstract}
In this article generalized attenuated ray transforms (ART)
and integral angular moments are investigated. Starting
from the Radon transform, the attenuated ray transform and
the longitudinal ray transform, we derive the concept of
ART-operators of order $k$ over functions defined on the
phase space and depending on time. The ART-operators are generalized for
complex-valued absorption coefficient as well as weight functions of 
polynomial and exponential type. Connections between ART operators of
various orders are established by means of the application of the
linear part of a transport equation. These connections lead
to inhomogeneous differential
equations of order $(k+1)$  for the ART of order $k$. Uniqueness theorems
for the corresponding boundary-value and initial boundary-value problems 
are proved. Properties of
integral angular moments of order $p$ are
considered and connections between the moments of
different orders are deduced. A close connection of the considered
operators with mathematical models
for tomography, physical optics and integral geometry
allows to treat the inversion of ART of
order $k$ as an inverse problem of determining the right-hand 
side of a corresponding differential equation.
\end{abstract}



\section{Introduction and preliminaries}

Many linear integral operators arising as mathematical models in computer and
emission tomography, wave optics, integral geometry of
tensor fields, can be classified as generalized
attenuated ray transforms (ART) that act on functions or
tensor fields.

In recent years a significant progress of emission tomography in biology and medicine
diagnostics can be observed, see~\cite{BudGulHue79,Natt86}. In
contrast to transmission computer
tomography, the mathematical setting of emission tomography problem
contains, in general, two unknown functions that have to
be reconstructed. The first function (absorption coefficient)
characterizes the absorption within the medium whereas the second
describes the distribution of internal sources whose
radiation is measured by detectors. The goal is to find the
distribution of internal sources $f$ and/or the absorption
coefficient $\varepsilon$ by given values of the attenuated
ray transform
 \begin{equation}\label{ART}
    I = \int\limits_{L} f(q) \exp \left(
    -\int_{L(q)} \varepsilon (p)\,dp
    \right) dq,
 \end{equation}
where $L(q)$ is the segment of the straight line $L$ between
a source point $q$ and the detector. In most mathematical settings of
emission tomography the absorption coefficient is
supposed to be known. Absorption phenomena also
arise in models of vector
tomography~\cite{KazBukh2007,Ainsw2013,FMon2016-1}. The
authors there investigate important theoretical questions
and use these as starting point to handle some aspects of vector tomography.

Another reason for investigating generalized ARTs are motivated by
the research area of physical optics, more specifically wave optics and
photometry.

We specifiy the mathematical models. Let a rectangular Cartesian coordinate system be
given in the Euclidean space $\RR^3$ with inner product $
\langle x, y \rangle $ and norm $| x |$ of elements $x =
(x^1, x^2, x^3)$, $y=(y^1, y^2, y^3) \in \RR^3$. We
denote the unit sphere by $\SS^2$, and 
$\SS \RR := \RR^3 \times \SS^2 = \{ (x,
\xi)\, |\, x \in \RR^3, \xi \in \SS^2 \}$. For a bounded
convex domain $D \subset \RR^3$ with smooth boundary
$\partial D$ we define $\SS D := \{
(x, \xi)\, |\, x \in D, \xi \in \SS^2 \}$. 
The sets $\SS \RR$ and $\SS D$ are known as the
spherical bundles for $\RR^3$ and $D$. The set of pairs
$(x,\xi) \in \SS \RR$ with fixed $x$ is denoted as
$\SS^2_x$. We assume that the domain $D$ contains a
distribution $f(q)$, $q \in D$, of sources of a
monochromatic wave field. Using the notation of {\it the
optical system}~\cite{Kirt75}, which represents a mathematical
formalization of a device like a camera, see Fig.~\ref{OS},
leads to the formulation of the direct problem of wave optics
consisting in solving a boundary-value problem
for the Helmholtz equation satisfying discontinuous
boundary conditions of Kirchhoff type, see~\cite{BornWolf73},
\begin{equation}\label{Helm_eq}
\left\{
\begin{array}{rr}
\Delta u + k^2 u & =0\\
u|_{\Sigma} & = u_n\\
u|_{L \backslash \Sigma} & = 0,
\end{array}
\right.
 \end{equation}
as well as the Sommerfeld radiation condition,
 \begin{equation}\label{Zomm_cond}
  \frac {\partial u}{\partial |x|} - \i ku
  = o \left( \frac {1}{|x|} \right) \quad
  \mbox{at}\,\, |x| \rightarrow \infty.
 \end{equation}
Here $u_n$ is a wave field hitting a diaphragm $\Sigma $,
$\i$ is the imaginary unit and $k$ is the wave number.
 \begin{figure}[h!]
  \centering
  \includegraphics[width=9cm,height=4.5cm]{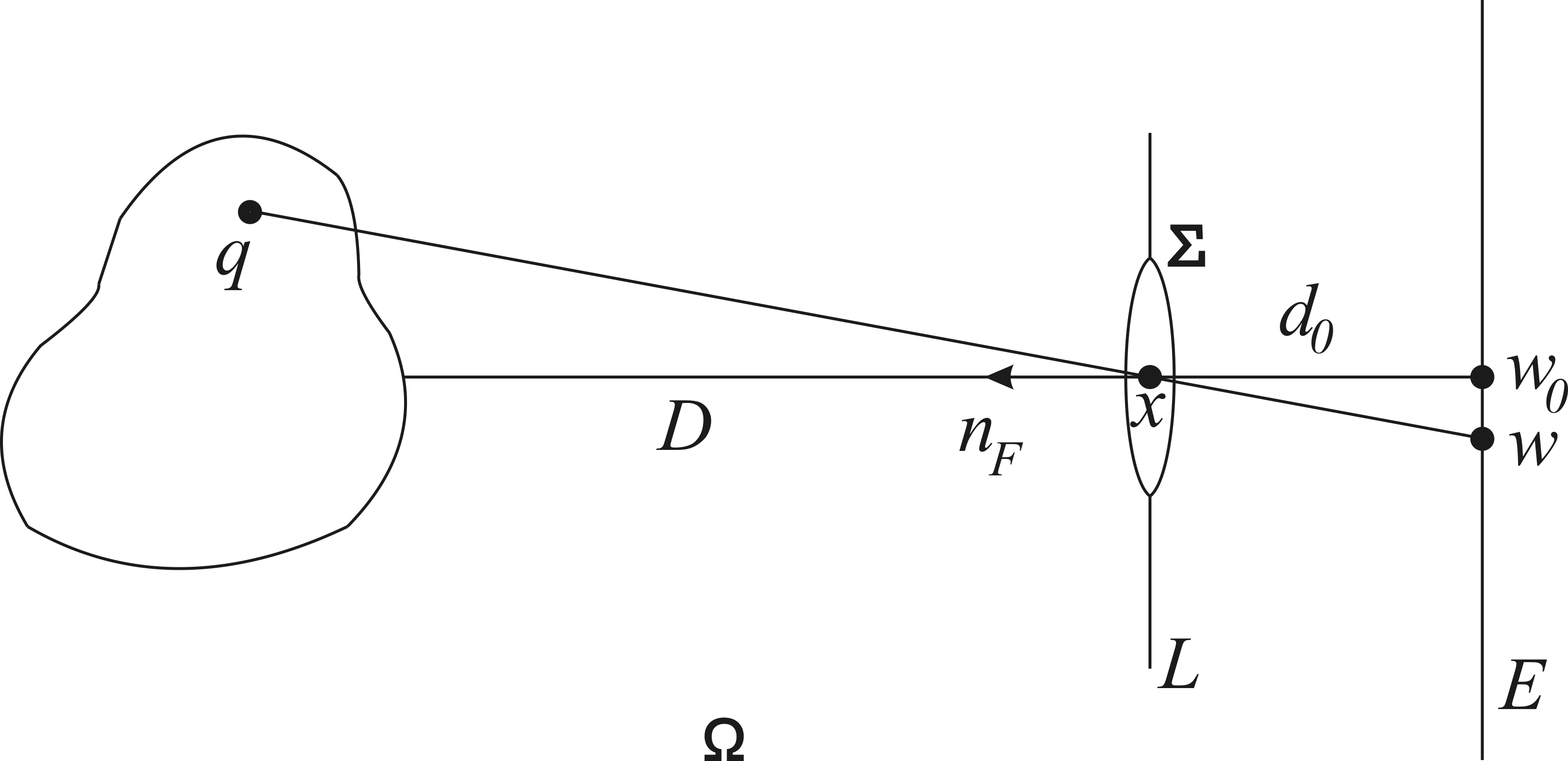}
  \caption{\small Scheme of optical system.}\label{OS}
 \vspace*{-3mm}
 \end{figure}
%
The optical system $\Omega$ depicted in Fig.~\ref{OS} consists of the plane $L$,
in which the diaphragm $\Sigma$ with center $x$ is located, $n_F$ denotes
the outer unit normal vector to $\Sigma$, $E$ is a screen with center $w_0$,
$d_0$ is the focal distance between $L$ and $E$ and $w$ the geometrical
image of a point $q\in D$.
An application of Green's function for the half-space yields
a solution $u$ of the direct problem for (\ref{Helm_eq}) 
represented as a Kirchhoff integral~\cite{BornWolf73}. Usually in
optics the Fraunhofer approximation for the far field is
used~\cite{Goodman70} which allows to simplify the
solution significantly. The obtained
approximate solution $u$ for (\ref{Helm_eq}) is then
represented as a convolution $u = \widetilde{u}_{\delta}
\ast \Delta$ with known kernel $\Delta$. Function
$\widetilde{u}_{\delta}$ is called {\it the ideal wave
image}~\cite{Kirt83} and has the representation
 \begin{equation}\label{Ideal_wave_im1}
  \widetilde{u}_{\delta}(x,\xi)
  = \int\limits_{0}^{\infty}
  s {e^{iks}} f (x-s\xi)
  d s.
 \end{equation}
A similar setting of the direct problem with incoherent sources
in $D$ and given constant absorption coefficient
$\varepsilon > 0$ leads to so called {\it ideal
photometric image}
 \begin{equation}\label{Ideal_phot_im0}
  \widetilde{u}_{\delta}(x,\xi) =
  \int\limits_{0}^{\infty}
  {e^{-\varepsilon s}} f (x-s\xi)
  d s.
 \end{equation}
The inverse problems of wave optics and photometry are
formulated as reconstruction problems for a function
$f$ which describes distributions of sources of a
monochromatic wave field or distributions of incoherent
sources. The initial data for these inverse problems are
integrals along all straight lines of $\RR^3$ in the
right-hand parts of (\ref{Ideal_wave_im1}),
(\ref{Ideal_phot_im0}). In other words we use the
integrals with lower limit $-\infty$ as initial
data for the corresponding inverse problems.

Tensor tomography traditionally has applications to the
problems of photo\-elas\-ticity and fiber
optics~\cite{AbPuro97}. New approaches and achievements emerge
in diffractive tomography of strains~\cite{LionWil14},
diffusion MRI-tomography and cross-polarization optic
coherent tomography~\cite{PaZeDefGu02,GelGel06}. 
Success of tensor tomography in studying of anisotropic
objects and materials in physics, geophysics, biology and
medicine has deep impact and is closely connected with the
progress in integral geometry of tensor fields, wherein
many types of ray transforms are
investigated~\cite{Shar94}, see furthermore
\cite{DerPol2014,SvDeVoSch2014,DerSve2015,DerMal2015}
for the 2D-case and \cite{FMon2016-2,FMon2018-1}
for arbitrary dimensions of Euclidean space or Riemannian manifolds.

As initial data for inverse problems in integral geometry connected to
tensor fields are in particular reprsented by the longitudinal
ray transform
 \begin{equation}\label{LRT}
  {\cal P } w (x,\xi) = \int\limits^{\infty}_{0}
  w_{i_{1} \ldots i_{m}}(x-s \xi)\xi^{i_{1}} \ldots
  \xi^{i_{m}} d s,
 \end{equation}
where $w_{i_{1} \ldots i_{m}}(x)$ is a symmetric tensor
field of rank $m$ ($m$-tensor field), $\xi^l$ are the
components of the vector of direction $\xi$, $|\xi|=1$, of a
straight line $L$ along which the integration is computed. 
Here and subsequently Einstein's summation rule is used, which says that
by repeating super- and subscripts in a monomial a
summation from $1$ to $n$ is meant ($n$ is the dimension of the
Euclidean space). The aim is to recover a
tensor field $w$ from given values of its longitudinal ray
transform (\ref{LRT}).

In this article the operators of ART 
for functions and tensor fields are generalized
in four different respects. At first, the functions $f(x)$ and
$\varepsilon(x)$ depend not only on the spatial
variable $x$, but also on the unit vector of direction $\xi$.
Secondly, the absorption coefficient $\varepsilon$
is a complex-valued function similar to those arising
in inverse scattering problems when Rytov's approach is applied~
\cite{MueKavWad79,BalJohSte80}, and then in the framework of
diffraction tomography, see, e.g.,\cite{Natt86}.
Next we take into account the concept of polynomially
weighted ray transform of tensor fields as
represented in~\cite{Shar94}. The final
generalization is connected with settings of dynamic
tomography~\cite{SchmLou02I,SchmLou02II,HahLou12} and
consists in consideration the situation that the internal sources
depend on time $t$. In other words we suppose the function
$f(t,x,\xi)$ to depend on time $t$, space $x$ and vector of direction
$\xi$. We use the notation $\alpha(x,\xi)$ for a
complex-valued function $\varepsilon(x,\xi) + \i
\rho(x,\xi)$.

{\it The attenuated ray transform (ART) of order $k$} is
defined by 
 \begin{equation}\label{TF_k-order}
  u_{k} (x,\xi)
  \displaystyle{ = \int\limits^{\infty}_0 s^k
  \exp \left( -\int^s_0
  \alpha(x-\sigma \xi,\xi) d \sigma \right)
  f(x-s\xi,\xi) d s.
  } 
 \end{equation}
Functions $\varepsilon \geq 0$, $\rho$ and $f$ are finite
and bounded with respect to the variable $x$. In
particular, 
the function $f(x,\xi)$ may have a form as
 \begin{equation}\label{Part-TFk}
  f(x,\xi) = w_{i_{1} \ldots i_{m}}(x)
  \xi^{i_{1}} \ldots \xi^{i_{m}}.
 \end{equation}
coinciding with the integrand in (\ref{LRT}).

Formula (\ref{TF_k-order}) defines a stationary ART of
order $k$. Suppose that $f$ is a function depending on time
$t$, too, and the propagation speed of
perturbation be, for simplicity, equal to unity. Then,
\begin{equation}\label{TF_k-order-t}
  u_{k}(t,x,\xi)
  \displaystyle{ = \int\limits^{\infty }_0 s^k
  \exp \left( -\int^s_0
  \alpha(x-\sigma \xi,\xi) d \sigma \right)
  f(t-s,x-s\xi,\xi) d s,
  } 
 \end{equation}
is called the {\it non-stationary ART of order $k$}. In
particular, the function $f(t,x,\xi)$ may be defined as
$$
  f(t,x,\xi) = w_{i_{1} \ldots i_{m}}(t,x)
  \xi^{i_{1}} \ldots \xi^{i_{m}}.
$$

Inversion formulas for the Radon transform ${\mathcal{R}}
\varphi$ of a function $\varphi(x)$, such as
 \begin{equation}\label{Gl0-Rad_Tr}
    {\mathcal{R}}\varphi (\eta,s)
    = \int\limits_{R^3}
    \varphi(x) \delta(\langle \eta,x \rangle - s) d x,
 \end{equation}
where $\delta(\cdot)$ denotes the $\delta$-function on $R$ and $\eta$
is the unit normal vector to the plane of integration,
do often involve the back-projection operator
${\mathcal{R}}^{\#} $, acting on the Radon transform,
 \begin{equation}\label{Gl0-BP_oper}
   {\mathcal{R}}^{\#}g (x)
   = \frac{1}{4 \pi^2}
   \int\limits_{\SS^2_x} g(\eta, \langle \eta,x \rangle)
   d \lambda_x(\xi), 
 \end{equation}
where $g(\eta,s)$ $=$ ${\mathcal{R}}\varphi (\eta,s)$ and
$d\lambda_x(\xi)$ is the angular measure on the sphere
$\SS^2_x$.

The back-projection operator acting on the longitudinal
ray transform (\ref{LRT}) is represented by
 \begin{equation}\label{Gl0-Inv-PV-m(j)}
  {\mathcal{P}}^{\#}
  h (x)
  = \frac{1}{4 \pi^2} \int\limits_{\SS^2_x}
  \xi^{i_{1}} \ldots \xi^{i_{m}}
  h(x,\xi)\,
  d \lambda_x(\xi),
 \end{equation}
where $h(x,\xi)$ $=$ ${\mathcal{P}}w (x,\xi)$ (see
formula (\ref{LRT})). The image of the back-projection
operator (\ref{Gl0-Inv-PV-m(j)}) is a symmetric tensor
field $\mu_{m}^{i_1 \ldots i_m}$ of rank $m$, $\mu_{m} :=
{\mathcal{P}}^{\#} {\mathcal{P}} w (x)$.

In the articles
\cite{DerSve2015,DerMal2015,DerMalSve2015-2,DerMalSve2018}
integral angular moments in 2D have been introduced as
generalizations of back-projection operators. 
The suggested operators, acting on arbitrary
tomographic transforms, are of utmost importance for 
recovering the singular support of a tensor field
by its known ART similar to (\ref{TF_k-order}). The
mathematical model is convenient for a medium with
refraction and variable absorption coefficient which may
be unknown. In this article the integral angular moments are
defined in a more general way then in 2D only. They map a
generalized ART of order $k$, $u_k$, with respect to functions $f(x,\xi)$ or
$f(t,x,\xi)$ to 3D-symmetric tensor fields.

As mentioned before the definition of the integral
angular moment ${\cal E}_{kp}^{i_1\dots i_p}(u)$ of ART
$u_k$ of order $k$ can be seen as a generalization of the back-projection
operator. It is defined by
 \begin{equation}\label{Ang_p-moment0}
  {\cal E}_{kp}^{i_1\dots i_p} u (x)
  = \int\limits_{\SS^2_x} u_k(x,\xi)
  \xi^{i_1} \dots \xi^{i_p}\,
  d \lambda_x(\xi), \quad p \geq 0.
 \end{equation}
In particular, if $f(x,\xi)$ is given as in (\ref{Part-TFk}), then
the integral angular moment (\ref{Ang_p-moment0})
coincides with the back-projection operator of the longitudinal
ray transform.

{\it Outline.} The first section of the article deals with certain connections
between generalized ART of different orders. We deduce
differential equations of order $k+1$, whose solutions
are the ART of order $k$ acting on a function $f(x,\xi)$
or, in the general case, on a $m$-tensor field $w(x)$.
Differential equations of the first order coincide with
stationary and non-stationary transport equations with
right-hand side $f(x,\xi)$, complex-valued
absorption coefficient $\alpha(x,\xi)$, and without the integral
part that describes the
scattering~\cite{KeizZwa72}. The second section is devoted to
proofs of uniqueness theorems for the emerging boundary-value and
initial boundary-value problems of arbitrary order. The third section contains
properties
of the integral angular moments
over ART of order $k$. Connections of certain
differential operators and the operators of divergence
with the operators of integral angular moments are
established. We conclude the article by emphasizing the connections
between generalized ART and different mathematical models
of tomography, physical optics and integral geometry and
an outlook to future research in the field.


\section{Main equations}

We start by defining differential equations whose solutions are 
attenuated ray transforms (ART)
of order $k$ of functions $f(x,\xi)$.

The operator ${\mathcal H}$, acting on functions $\psi(x, \xi)$ that are differentiable on
$\SS \RR$, is defined in
invariant form by
 \begin{equation}\label{H_def}
  {\mathcal H}\psi (x,\xi)
  = \frac{d}{d \tau} \psi(x+\tau \xi,\xi)
  \Big|_{\tau=0}.
 \end{equation}
In a rectangular Cartesian coordinate system the operator
is represented as
 \begin{equation}\label{H_coord}
  {\mathcal H}\psi (x,\xi) =
  \xi^1 \frac{\partial \psi}{\partial x^1}
  + \xi^2 \frac{\partial \psi}{\partial x^2} + \xi^3
  \frac{\partial \psi}{\partial x^3}
  = \langle \xi\,, \nabla \psi \rangle
  = {\rm div}\,(\psi \xi).
 \end{equation}

 \begin{lemma}\label{lemma1}
Let $\varepsilon(x,\xi)\ge 0$, $\rho(x,\xi)$, $f(x,\xi)$
be elements of $C^1(\SS D)$. Then for $u_k(x,\xi)$
defined by \eqref{TF_k-order}, $k \geq 1$ being an integer, the identity
 \begin{equation}\label{H_eq_1}
  \big( ({\mathcal H} + \alpha)u_{k} \big)(x,\xi)
  = k\,u_{k-1}(x,\xi)
 \end{equation}
is valid.
 \end{lemma}

 \begin{proof}
We prove relation (\ref{H_eq_1}) directly. By
definition (\ref{H_def}),
$$
 \begin{array}{rcl}
  {\mathcal H}u_{k} (x,\xi)
&  = & {\displaystyle\frac{d}{d \tau} u_k(x + \tau \xi,\xi)
  \Big|_{\tau=0} }  \\[3ex]
& = & {\displaystyle \int\limits^{\infty }_0
  \hspace*{-1mm} s^k \left.\frac{d}{d \tau} \exp \left(
  \hspace*{-1mm} - \hspace*{-1mm}
  \int^s_0 \alpha(x
  + (\tau - \sigma)\xi,\xi)
  d \sigma \right) \right|_{\tau=0}
  f (x - s \xi,\xi) d s } \\[4ex]
&& {\displaystyle + \int\limits^{\infty }_0 \hspace*{-1mm}
  s^k \exp \left(
  \hspace*{-1mm} - \hspace*{-1mm}
  \int^s_0\alpha(x - 
  \sigma \xi,\xi) d \sigma \right)  \frac{d}{d\tau }
  f \big(x + (\tau-s) \xi,\xi\big)  \Big|_{\tau=0}  d s }.
 \end{array}
$$
For fixed $x$ and vector $\xi$ the functions
$f(x+v\xi,\xi)$ and $\alpha(x+v\xi,\xi)$ can be treated
as functions depending only on a real variable $v$,
$$
  f (x+v\xi) = \varphi(v), \quad \alpha(x+v\xi,\xi)
  = \psi(v).
$$
Then,
$$
  \frac{d}{d \tau} \varphi(\tau - s)
  = -\frac{d}{ds} \varphi(\tau - s), \quad
  \frac{d}{d \tau} \psi(\tau-\sigma)
  = -\frac{d}{d \sigma} \psi(\tau - \sigma),
$$
and hence
$$
 \begin{array}{rcl}
  {\mathcal H}u_{k} (x,\xi)
& = & \displaystyle{\int\limits^{\infty }_0
  s^k \varphi(-s)
  \exp \left( -\int^s_0
  \psi(-\sigma) d \sigma \right)
  \int^s_0 \frac{d}{d \sigma}
  \psi(\tau - \sigma)\Big|_{\tau=0} d \sigma \,
  d s }  \\[4ex]
& & \displaystyle{ - \int\limits^{\infty}_0 s^k
  \exp \left(
  - \int^s_0 \psi(-\sigma) d \sigma \right)
  \left. \frac{d \varphi(\tau-s)}{ds}
  \right|_{\tau=0} d s }.
 \end{array}
$$
Passing to the limit $\tau \rightarrow 0$ and applying
integrating by parts in the second term of the right-hand part yield
 \begin{equation}\label{Diff_Hu_m}
 \begin{array}{rcl}
  {\mathcal H}u_{k} (x,\xi)
  &  = & \displaystyle{\int\limits^{\infty }_0
  s^k \exp \left( -\int^s_0
  \psi(-\sigma) d \sigma \right)
  \varphi(-s) \big[ \psi(-s) - \psi(0)
  \big] d s }\\[4ex] 
 && \displaystyle{} - s^k \exp
  \left( -\int^s_0 \psi(-\sigma)
  d \sigma \right) \varphi(-s)
  \Big|^{\infty}_{s=0}  \\[4ex]
 && {}+\displaystyle{\, k \int\limits^{\infty }_0
  s^{k-1} \exp
  \left( -\int^s_0 \psi(-\sigma)
  d \sigma \right) \varphi(-s) d s } 
  \vspace*{1mm} \\[4ex]
 && {} -\displaystyle{ \int\limits^{\infty}_0 s^k \exp
  \left( - \int^s_0
  \psi(-\sigma) d \sigma \right)
  \varphi(-s) \left[ \frac{d}{ds}
  \int^s_0
  \psi(-\sigma) d \sigma \right] d s }.
 \end{array}
 \end{equation}
Since the function $f(x,\xi)$ is finite with respect to $x$,
$s^k = 0$ at $k \geq 1$ and $s=0$, the second term vanishes for $s
\rightarrow \infty $ and $s=0$. The third term
on the right-hand side is equal to $k u_{k-1}(x,\xi)$ according to
(\ref{TF_k-order}). Summing up the first and the last
items we obtain $-\alpha(x,\xi) u_{k}(x,\xi)$. Thus,
$$
  {\mathcal H}u_{k} (x,\xi)
  = k u_{k-1} (x,\xi) - \alpha(x,\xi)
  u_{k}(x,\xi),
$$
and the equality (\ref{H_eq_1}) is proved.
 \end{proof}
 
\medskip

 \begin{lemma}\label{lemma2}
Let $\varepsilon(x,\xi)\ge 0$, $\rho(x,\xi)$, $f(x,\xi)$
be elements of $C^1(S D)$. Then the
function $u_{0}(x,\xi)$ defined by \eqref{TF_k-order} is a solution
of the equation
 \begin{equation}\label{H_eq_RP=mu}
  ({\mathcal H} + \alpha)u_{0} (x,\xi)
  = f(x,\xi).
 \end{equation}
 \end{lemma}

 \begin{proof}
Arguments that are similar to those of the proof of lemma
\ref{lemma1} lead to
following equality,
 \begin{equation}\label{Diff_Hu_m_0}
 \begin{array}{rcl}
  {\mathcal H}u_{0} (x,\xi)
  & = & \displaystyle{ \int\limits^{\infty }_0
  \varphi(-s) \exp \left( -\int^s_0
  \psi(-\sigma) d \sigma \right)
  \psi(-s) d s }\\[4ex] 
 && {} -\displaystyle{ \int\limits^{\infty }_0
  \varphi(-s) \exp \left( -\int^s_0
  \psi(-\sigma) d \sigma \right)
  \psi(0) d s }\\[4ex]
 && {} - \displaystyle{\, \exp
  \left( -\int^s_0 \psi(-\sigma)
  d \sigma \right) \varphi(-s)
  \Big|^{\infty}_{s=0} } \\[4ex]
 && {} - \displaystyle{ \int\limits^{\infty}_0
  \varphi(-s) \exp
  \left( - \int^s_0
  \psi(-\sigma) d \sigma \right)
  \left[ \frac{d}{ds}
  \int^s_0
  \psi(-\sigma) d \sigma \right] d s }
 \end{array}
 \end{equation}
instead of (\ref{Diff_Hu_m}). It is easy to see that the first and the last terms on the right-hand side
differ only by signs and so their sum is equal
to zero. The second term is equal to $-\psi(0)\,u_{0}
(x,\xi)$, and the third is equal to $\varphi(0)$. As
$\varphi(0) = f(x,\xi)$, $\psi(0) = \alpha(x,\xi)$ we have
${\mathcal H} u_{0}(x,\xi) = f(x,\xi)
- \alpha (x,\xi) u_0(x,\xi)$ and the statement of
lemma \ref{lemma2} is verified.
 \end{proof}

We consider the important partial case if the function
$f(x,\xi)$ in (\ref{TF_k-order}) has a form
(\ref{Part-TFk}) which is a $m$-homogeneous polynomial in $\xi^j$.
The set of all given symmetric $m$-tensor
fields $w(x)=(w_{i_1 \ldots i_m}(x))$,
$i_1,\ldots,i_m=1,2,3$, in $\RR^3$ or in $D$ is denoted by $S^m(D)$. For simplicity we usually
write $S^m$. The scalar product
in $S^m$ is defined by 
 \begin{equation}\label{Gl0-In_Pr-Sm}
  \langle u(x) , v(x) \rangle
  = u_{i_1 \ldots i_m}(x) v^{i_1 \ldots i_m}(x),
 \end{equation}
and we use the notation $\langle w , \xi^m
\rangle$ for the sum $w_{i_{1} \ldots i_{m}} \xi^{i_{1}}
\ldots \xi^{i_{m}}$. In
Euclidean space with rectangular Cartesian coordinate
system there is no difference between contravariant and
covariant components of tensors. In this article we usually use covariant
components of tensors.

\medskip

 \begin{corollary}\label{corol1}
Let $\varepsilon(x,\xi)\ge 0$, $\rho(x,\xi)$, $f(x,\xi)$
be elements of $C^1(\SS D)$. Let the function
$f(x,\xi)$ in \eqref{TF_k-order} be of a form $\langle
w(x) , \xi^m \rangle$ \eqref{Part-TFk}, where $w(x)$, $x
\in D$, is a symmetric $m$-tensor field, $m
\geq 0$ being an integer. Then $u_0(x,\xi)$ is a solution of
the equation
 \begin{equation}\label{TF_H_eq_1}
  \big( ({\mathcal H} + \alpha)u_{0} \big)(x,\xi)
  = \langle w , \xi^m \rangle.
 \end{equation}
 \end{corollary}

 \begin{proof}
Since lemma \ref{lemma2} is valid for any function
$f(x,\xi)$, it also holds for each component
$w_{i_{1} \ldots i_{m}}(x)$ of symmetric $m$-tensor field
$w(x)$. We take into account now the linearity of the inner product and get
 \begin{equation}\label{Lin_Prop_ART}
 \begin{array}{c}
  \displaystyle{ \int\limits^{\infty}_0 
  s^k \exp \left( 
  - 
  \int^s_0 
  \alpha(x-\sigma \xi, \xi) d \sigma \right)
  \langle w(x-s\xi)\, ,\, \xi^m \rangle d s } \\[4ex]
  \quad = \displaystyle{ 
   \left \langle \xi^m \,, 
  \,\int\limits^{\infty}_0 
  s^k \exp \left( 
  - 
  \int^s_0 
  \alpha(x-\sigma \xi, \xi) d \sigma \right)
  w(x-s \xi) d s \right \rangle }.
 \end{array}
 \end{equation}
This gives (\ref{TF_H_eq_1}).
 \end{proof}

We define an operator $\L_k$ $:$ $C^k(\SS \RR)$ $\rightarrow$
$C(\SS \RR)$ by
$$
 \hspace*{-7mm}
  \L_1 \psi (x,\xi)
  = ({\mathcal H} + \alpha)\psi (x,\xi),
$$
$$
 \hspace*{10mm}
  \L_k \psi (x,\xi)
  = \frac{1}{k-1}({\mathcal H} + \alpha)
  \L_{k-1} \psi (x,\xi),
$$
for $k>1$ being an integer.

\medskip

 \begin{theorem}\label{theorem1} Let $\varepsilon(x,\xi)
\geq 0$, $\rho(x,\xi)$, $f(x,\xi)$ 
be elements of $C^{k+1}(\SS D)$, $w(x) \in C^{k+1}(S^m)$,
$x \in D$, be a symmetric $m$-tensor field. Then for
an integer $k \geq 0$, the function $u_k(x,\xi)$ is a
solution of the equation
 \begin{equation}\label{L_eq_RP=mu}
  \L_{k+1}u_{k} (x,\xi)
  = f(x,\xi)
 \end{equation}
In particular for $f(x,\xi) = \langle w(x), \xi^m
\rangle$, the equality for $m \geq 0$
 \begin{equation}\label{L_eq_RP=TF-m}
  \L_{k+1}u_{k} (x,\xi)
  = \big \langle w(x) , \xi^m \big\rangle
 \end{equation}
is valid with $u_{k}(x,\xi)$ defined by
\eqref{TF_k-order}.
 \end{theorem}

 \begin{proof} Under the assumptions of lemma \ref{lemma1},
we apply on the both sides of (\ref{H_eq_1}) $(k-1)$ times
the operator $({\mathcal H} + \alpha)$. This gives
$({\mathcal H} + \alpha)^k u_{k}(x,\xi) = u_{0}(x,\xi)$. 
Applying the operator $({\mathcal H} + \alpha)$ once more and using lemma \ref{lemma2}, we get
the first statement (\ref{L_eq_RP=mu}) of the theorem.
Applying these arguments to $f(x,\xi) = \langle w(x),
\xi^m \rangle$ we get formula (\ref{L_eq_RP=TF-m}).
 \end{proof}

Thus we have differential
relations (\ref{L_eq_RP=mu}), (\ref{L_eq_RP=TF-m}),
connecting ART $u_{k}(x,\xi)$ of order $k$ with the
function $f(x,\xi)$ or symmetric $m$-tensor field $w(x)$.

We advance to the non-stationary case.

\medskip

 \begin{lemma}\label{lemma3} For functions
$\varepsilon(x,\xi) \geq 0$, $\rho(x,\xi)$ $\in$ $C^1(\SS
D)$, $f(t,x,\xi)$ $\in$ $C^1(R \times \SS D)$. Then for
$u_{k}(t,x,\xi)$ determined by
\eqref{TF_k-order-t}, $k \geq 0$, the
equalities
 \begin{equation}\label{H_eq_1-t}
  \hspace*{4mm} \Big( \frac{\partial}{\partial t}
  + {\mathcal H} + \alpha \Big)
  u_{k} = k\, u_{k-1}, \quad k \geq 1\,,
 \end{equation}
 \begin{equation}\label{H_eq_RP=mu-t}
  \hspace*{-4mm}
  \Big( \frac{\partial}{\partial t}
  + {\mathcal H} + \alpha \Big)
  u_{0} = f 
 \end{equation}
are valid. In particular, if $f(t,x,\xi) = \langle
w(t,x), \xi^m \rangle$ for a symmetric $m$-tensor field
$w(t,x) \in C^1(\RR \times S^m)$, then the
equality
 \begin{equation}\label{H_eq_RP=TF-m-t}
  \Big( \frac{\partial}{\partial t}
  + {\mathcal H} + \alpha \Big)
  u_{0} =
  \big\langle w , \xi^m \big\rangle
 \end{equation}
holds true.
 \end{lemma}

 \begin{proof}
Fixing non-stationary ART
$u_k(t,x,\xi)$ of order $k$ defined by (\ref{TF_k-order-t}),
we find $\displaystyle{ \frac{\partial
u_{k}}{\partial t}(t,x,\xi)}$. Since only function $f$
depends on $t$, we have to compute $\displaystyle{
\frac{\partial f(t-s,x-s\xi,\xi)} {\partial t} }$,
\begin{equation}\label{um-t-deriv}
  \frac{\partial u_{k}}{\partial t}
  = \int\limits^{\infty}_0
  s^k \exp \left( -\int^s_0
  \alpha(x-\sigma \xi,\xi) d \sigma \right)
  \frac{\partial f (t-s,x-s\xi,\xi)}{\partial t} d s.
\end{equation}
We calculate the total derivative $\displaystyle{ \frac{d
f}{d s} }$ at first,
$$
  \frac {d f }{d s} = \left\langle
  \frac{\partial f(t-s, x-s\xi,\xi)}
  {\partial (x-s\xi)},\;
  \frac{\partial (x-s\xi)}{\partial s} \right\rangle
  + \frac{\partial f(t-s, x-s\xi,\xi)}{\partial (t-s)}
  \,\frac{\partial (t-s)}{\partial s}.
$$
Using
$$
\,\displaystyle{ \frac{\partial f(\theta,y,\xi)}
  {\partial y}} =
  \nabla_{y} f(\theta,y,\xi)
$$
for $y = x-s\xi$, $\theta=t-s$, and
$\,\displaystyle{\frac{\partial y}{\partial s}=-\xi}$,
the derivative is represented as
$$
  \frac {d f }{d s} = \left\langle
  \nabla_{y} f(\theta,y,\xi)\,, -\xi \right\rangle
  + \frac{\partial f(\theta,y,\xi)}{\partial \theta}
  \,\frac{\partial \theta}{\partial s}.
$$
The next step consists in calculation the result of ${\mathcal H}$ action on $f$,
$$
  {\mathcal H} f \equiv
  \displaystyle{ \frac{d f(\theta,y+\tau\xi,\xi)}
  {d \tau} \Big|_{\tau=0} }
  =  \displaystyle{
  \left\langle \frac{\partial f(\theta,y+\tau\xi,\xi)}
  {\partial (y+\tau \xi)}\,,\,
  \frac{\partial (y+\tau\xi)}
  {\partial \tau} \right\rangle \Big|_{\tau=0}}
   =  \displaystyle{ \left\langle
  \nabla_{y} f(\theta,y,\xi)\,, \xi \right\rangle }.
$$
Because of $\displaystyle{ \frac{\partial \theta}{\partial s} }$
$=$ $\displaystyle{ - \frac{\partial \theta}{\partial t}
}$ we obtain
$$
   \displaystyle{ \frac{\partial f(\theta,y,\xi)}
   {\partial t} } 
   = \displaystyle{ -\frac{d f(\theta,y+\tau\xi,\xi)}
   {d \tau}
   \Big|_{\tau=0}
   - \frac{d f(\theta,y,\xi)}{d s}.}
$$
This implies
$$
 \begin{array}{rcl}
  \displaystyle{ \frac{\partial u_{k}}{\partial t}(t,x,\xi) }
  & = & \displaystyle{ - \hspace*{-1mm}
  \int\limits^{\infty}_0 \hspace*{-1mm} s^k
  \exp \left( - \int^{s}_0
  \alpha(x-\sigma\xi, \xi) d \sigma \right)
  \frac{d f(t-s, x+(\tau-s)\xi,\xi)}{d \tau}
  \Big|_{\tau=0} d s }
  \vspace*{2mm} \\[4ex]
 & & {} -\displaystyle{ \hspace*{-1mm}
  \int\limits^{\infty}_0 \hspace*{-1mm} s^k
  \exp \left( - \int^s_0 \alpha(x-\sigma \xi,\xi)
  d \sigma \right)
   \frac{d f(t-s, x-s\xi,\xi)}{d s} \, d s.}
 \end{array}
$$
Based on the definition of operator ${\mathcal H}$ we
integrate the second term of the right-hand side of the
last expression and obtain
$$
 \begin{array}{rcl}
  \displaystyle{ \Big( \frac{\partial u_{k}}{\partial t} }
   \displaystyle{ +{\mathcal H} u_{k} \Big)(t,x,\xi) }
 & = & \displaystyle{ \int\limits^{\infty}_0 s^k
   \left.\frac{d}{d \tau}
  \exp \left( - \int^{s}_0
  \alpha(x + (\tau -\sigma)\xi,\xi) d \sigma \right)
  \right|_{\tau=0} }
  \,f(t-s, x-s\xi, \xi) d s \\[4ex]
&& {} -  \displaystyle{ \,s^k \exp
  \left( -
   \int^s_0 \alpha(x-\sigma \xi,\xi)
   d \sigma \right) f(t-s, x-s\xi, \xi)
   \Big|^{\infty }_{s=0} } \\[4ex]
& & {}  + \displaystyle{
  \int\limits^{\infty}_0
  f(t-s, x-s\xi, \xi) \frac{d}{d s}
  \left[ s^k \exp
  \left( 
  - \int^s_0
  \alpha(x-\sigma \xi,\xi) d \sigma \right)
  \right]ds.}
 \end{array}
$$
We refer to the reasonings and calculations similar
to those in the proofs of lemmas \ref{lemma1} and
\ref{lemma2}. As before, taking into account
the different results depending on $k \geq 1$ or
$k=0$, we obtain
(\ref{H_eq_1-t}), (\ref{H_eq_RP=mu-t}) of lemma
\ref{lemma3} for non-stationary ART $u_k$
over the function $f(t,x,\xi)$. In partial case
$f(t,x,\xi) = \langle w(t,x),\xi^m \rangle$, where
$w(t,x)$ is a symmetric $m$-tensor field, the
equality (\ref{H_eq_RP=TF-m-t}) is verified instead of
(\ref{H_eq_RP=mu-t}). As it was before the linearity
property (\ref{Lin_Prop_ART}) has been applied.
 \end{proof}

\medskip

We define the operator ${\L}^t_k : C^k(R \times \SS D)
\rightarrow C(R \times \SS D)$ by mathematical induction
on $m$,
$$
 \begin{array}{rcl}
  \displaystyle{ {\L}^t_1 \psi (t,x,\xi) }
  &
  = 
  & \displaystyle{ \Big(
  \frac {\partial}{\partial t} + {\mathcal H} + \alpha
  \Big) \psi (t,x,\xi) }, \vspace*{3mm} \\
  {\L}^t_{k} \psi (t,x,\xi) 
  &
  = 
  & \displaystyle{ \Big( \frac{\partial}
  {\partial t} + {\mathcal H} + \alpha \Big)
  {\L}^t_{k-1} \psi (t,x,\xi) }, \quad k > 1.
 \end{array}
$$

\begin{theorem}\label{theorem2} For $\varepsilon(x,\xi)
\geq 0$, $\varepsilon(x,\xi), \rho(x,\xi)$ $\in$
$C^{k+1}(\SS D)$, $f(t,x,\xi) \in C^{k+1}(R \times \SS
D)$, the function $u_k(t,x,\xi$ defined by
(\ref{TF_k-order-t}) is a solution of the equation
 \begin{equation}\label{L_eq_RP=mu-t}
  ({\L}^t_{k+1}u_{k})(t,x,\xi)
   = f(t,x,\xi).
 \end{equation}
In particular, the function defined by (\ref{TF_k-order-t})
$u_k(t,x,\xi$ containing $f(t,x,\xi) = \langle w(t,x) ,
\xi^m \rangle$, where $w$ is symmetric $m$-tensor field,
is a solution of the equation
 \begin{equation}\label{L_eq_RP=TF-m-t}
   ({\L}^t_{k+1}u_{k})(t,x,\xi)
   = \big\langle w(t,x) , \xi^m \big\rangle.
 \end{equation}
 \end{theorem}

 \begin{proof} The proof is based
on lemma \ref{lemma3} and follows the lines of the proof of theorem
\ref{theorem1}.
 \end{proof}


\section{Uniqueness theorems}

We prove uniqueness theorems for boundary-value and
initial boundary-value problems of the equations
(\ref{L_eq_RP=mu}), (\ref{L_eq_RP=mu-t}), respectively.
We remind that $D\subset \RR^3$ is a bounded convex domain
with smooth boundary $\partial D$.

\medskip

 \begin{theorem}\label{theorem3} Let for
$\varepsilon(x,\xi), \rho(x,\xi) \in C^k( \SS D)$,
$\varepsilon(x,\xi) \geq 0$, a function $\varphi(x,\xi)
\in C^{k+1}(\SS D)$ be a solution of the equation
 \begin{equation}\label{L_eq_RP=mu_uniq}
  \frac{1}{k!} \big( {\mathcal H}
  + \alpha \big)^{k+1} \varphi
  = 0
 \end{equation}
and satisfies the boundary-value conditions
 \begin{equation}\label{Boun_Val_Stat}
  \varphi(x,\xi) = ({\mathcal H} \varphi)(x,\xi)
  = \dots = ({\mathcal H}^k \varphi)(x,\xi) = 0,
  \quad x \in \partial D, \quad  \langle n_x,
  \xi \rangle < 0,
 \end{equation}
where $n_x$ is the outer normal to the surface $\partial D$
at a point $x$. Then $\varphi(x,\xi) \equiv 0$ for
$(x,\xi) \in \SS D$.
 \end{theorem}

 \begin{proof} We prove the statement for $k=0$ at first.
For this case (\ref{L_eq_RP=mu_uniq}) reads as
$({\mathcal H} + \alpha) \varphi = 0$. Since the coefficient
$\alpha(x,\xi) = \varepsilon(x,\xi) + \i \rho(x,\xi)$ in
(\ref{Boun_Val_Stat}) is complex-valued, the
function $\varphi(x,\xi)$ is complex-valued too, and thus
it can be represented in a form $\varphi = \varphi_1 + \i
\varphi_2$. We write $({\mathcal H} + \alpha) \varphi$
in more details,
$$
  ({\mathcal H} + \varepsilon + \i \rho)
  (\varphi_1 + \i \varphi_2)
  = ({\mathcal H} \varphi_1 + \varepsilon \varphi_1
  - \rho \varphi_2)
  + \i ({\mathcal H} \varphi_2 + \varepsilon \varphi_2
  + \rho\varphi_1) = 0,
$$
and multiply the obtained formula with $\bar{\varphi} =
\varphi_1 - \i \varphi_2$. Here, the notations
$\bar{\varphi}$ for complex conjugate and $|\varphi|$ for
modulus of complex-valued function $\varphi$ are used.
Then,
$$
  Re \left\{ \bar{\varphi}({\mathcal H} + \bar{\alpha})
  \varphi \right\}
  = \frac{1}{2} {\mathcal H} (|\varphi|^2)
  + \varepsilon |\varphi|^2 = 0.
$$
After integration of the last equation over $D$ and the unit
sphere $\SS^2_x$, we get
$$
  \frac{1}{2} \int\limits_D \int\limits_{\SS^2_x}
  {\mathcal H} (|\varphi|^2) d x d \lambda_x(\xi)
  + \int\limits_D \int\limits_{\SS^2_x}
  \varepsilon |\varphi|^2 d x d\lambda_x(\xi) = 0,
$$
where $d \lambda_x(\xi)$ is the angular measure on $\SS^2_x$,
$x \in D$.

Because of ${\mathcal H} (|\varphi|^2) = {\rm div}(|\varphi|^2
\xi)$ (see (\ref{H_coord})) the Gauss-Ostrogradsky
formula can be applied to the first integral of last
expression. In this way we obtain
 \begin{equation}\label{Form_Gau-Ostr}
  \frac{1}{2} \int\limits_{\partial D}
  \int\limits_{\SS^2_x}
  \langle n_x,\xi \rangle |\varphi|^2 d s d\lambda_x(\xi)
  + \int\limits_D \int\limits_{\SS^2_x} \varepsilon
  |\varphi|^2 d \lambda_x(\xi) d x = 0.
 \end{equation}
The condition (\ref{Boun_Val_Stat}) for $k=0$, implies that
$\varphi(x,\xi)$ vanishes at $\langle n_x, \xi \rangle <
0$, $x\in \partial D$. Hence, the first integral at the
left-hand side of (\ref{Form_Gau-Ostr}) is equal to zero.
From this and the non-negativity of $\varepsilon(x, \xi)$
it follows that (\ref{Form_Gau-Ostr}) is valid if and
only if $\varphi(x,\xi) = 0$ for $(x,\xi) \in \SS D$. This yields 
the statement for $k=0$.

Assume further that the theorem holds true for some $j=k-1$, $j
\geq 1$. We prove it for $j=k$ (i.e. for the equation
of order $k+1$). To this end we consider the equation
$$
  {\L}_{k+1} \varphi = \frac {1}{k!}
  ({\mathcal H} + \alpha)^{k+1} \varphi
  = \frac{1}{k} ({\mathcal H}
  + \alpha)({\L}_k \varphi) = 0,
$$
and denote ${\L}_k \varphi $ as $A + \i B$. Then,
$$
  \frac{1}{k} ({\mathcal H} + \alpha) (A + \i B)
  = \frac{1}{k}
  ({\mathcal H} A + \varepsilon A - \rho B)
  + \frac{\i}{k}({\mathcal H} B + \varepsilon B
  + \rho A) = 0.
$$
We multiply the obtained expression with
$\overline{{\L}_k \varphi}$, the complex conjugate of ${\L}_k \varphi$. The
multiplication is possible by the induction
assumption, so
$$
 \begin{array}{rcl}
  \displaystyle{ (A - \i B) \frac{1}{k}({\mathcal H}
  + \alpha)
  (A + \i B) } \vspace*{2mm}
 & = & \displaystyle{ \frac{1}{k} \big( A {\mathcal H} A
  + B {\mathcal H} B
  + \varepsilon A^2
  + \varepsilon B^2 \big) } \\
  \vspace*{2mm} 
 & &
   \displaystyle{{}+ \frac{\i}{k} \big( A {\mathcal H} B
  - B {\mathcal H} A
  + \rho A^2 + \rho B^2 \big) = 0.}
 \end{array}
$$
Hence
\begin{align*}
&  \displaystyle{ Re \big\{ \big( \overline{{\L}_k \varphi}
  \big)\frac{1}{k}
  ({\mathcal H} + \alpha)({{\L}_k \varphi}) \big\} }
  = \displaystyle{ \frac{1}{2k} {\mathcal H} (A^2+B^2)
  + \frac{1}{k} \varepsilon (A^2+B^2) }
  \vspace*{2mm} \\
&  = \displaystyle{ \frac{1}{2k} {\mathcal H}
  \Big( \big| {\L}_k \varphi \big|^2
  \Big) + \frac{1}{k} \varepsilon \big| {\L}_k
  \varphi \big|^2 = 0. }
\end{align*}
After integrating the last expression over $D$ and
$\SS^2_x$, $x \in D$, and applying the
Gauss-Ostrogradsky formula we obtain an expression as
(\ref{Form_Gau-Ostr}), where, instead of $|\varphi|^2$, the
term $|{\L}_j \varphi|^2$ appears. The term contains
powers ${\mathcal H}^j \varphi$ of the operator
${\mathcal H}$ with $j\leq k$, so at $x \in \partial
D$, $\langle n_x, \xi \rangle < 0$, it follows that
${\L}_k \varphi = 0$. We can conclude, as in case
$k=0$, that $\varphi(x,\xi) = 0$ for $(x,\xi) \in \SS D$.
The proof of the theorem is complete.
 \end{proof}

\medskip

 \begin{theorem}\label{theorem4} Let for the functions
$\varepsilon(x,\xi), \rho(x,\xi)$ $\in$ $C^k(\SS D)$,
$\varepsilon(x,\xi) \geq 0$, a function
$\varphi(t,x,\xi)$ $\in$ $C^{k+1} (R \times S D)$ be a
solution of
the equation
 \begin{equation}\label{L_eq_RP=mu-t_uniq}
  \frac{1}{k!} \Big( \frac{\partial}{\partial t}
  + {\mathcal H} + \alpha \Big)^{k+1} \varphi
  = 0,
 \end{equation}
that satisfies the initial conditions
 \begin{equation}\label{Init_Val-t}
  \varphi(0,x,\xi) = \frac{\partial \varphi}
  {\partial t}(0,x,\xi)
  = \ldots = \frac{\partial^k \varphi}
  {\partial t^{k}}(0,x,\xi)
  = 0,
 \end{equation}
and, for $x \in \partial D$, $\langle n_x,\xi \rangle <
0$, $t \geq 0$, the boundary conditions
 \begin{equation}\label{Boun_Val-t}
 \hspace*{-2mm} \varphi(t,x,\xi)
  = ({\mathcal H} \varphi)(t,x,\xi)
  = \ldots
  = ({\mathcal H}^k \varphi )(t,x,\xi) = 0,
 \end{equation}
where $n_x$ is outer normal to the surface $\partial D$
at a point $x$. Then $\varphi(t,x,\xi) = 0$ for $t
> 0$, $(x,\xi) \in \SS D$.
 \end{theorem}

 \begin{proof}
We prove the theorem for $k=0$, i.e., for
the equation of the first order.

Considering $\varphi = \varphi_1 + \i \varphi_2$ and
multiplying both parts of the equality
$$
  \Big( \frac{\partial}{\partial t}
  + {\mathcal H} + \alpha \Big)
  \big( \varphi_1 + \i \varphi_2 \big) = 0
$$
by $\overline{\varphi} = \varphi_1 - \i \varphi_2$, we
obtain
$$
  \frac{1}{2} \frac{\partial}{\partial t}
  |\varphi|^2
  + \frac{1}{2} {\mathcal H} \big( |\varphi|^2 \big)
  + \varepsilon |\varphi |^2 = 0.
$$
We integrate the resulting expression by $t$ from $0$ to $T
\in (0,\infty)$, domain $D$ and sphere $\SS^2_x$, then
use Gauss-Ostrogradsky formula and obtain
$$
 \hspace*{-16mm} \frac{1}{2} \int\limits_D
  \int\limits_{\SS^2_x}
  \Big( |\varphi(T,x,\xi)|^2
  - |\varphi(0,x,\xi)|^2 \Big)
  d \lambda_x(\xi) d x
$$
 \vspace*{-1mm}
$$
 \hspace*{8mm} +\, \frac{1}{2} \int\limits^T_0
  \int\limits_{\partial D}
  \int\limits_{\SS^2_x} \langle n_x,\xi \rangle
  |\varphi|^2
  d \lambda_x(\xi) d S d t
  + \int\limits^T_0 \int\limits_D
  \int\limits_{\SS^2_x}
  \varepsilon |\varphi|^2 d \lambda_x(\xi)
  d x d t = 0.
$$
Taking into account that $T$ is arbitrary, the initial conditions
(\ref{Init_Val-t}) and boundary conditions (\ref{Boun_Val-t})
(for $k=0$), we justify that the last
formula is correct only if $\varphi(t,x,\xi) = 0$ for $t
\in (0, \infty)$, $(x,\xi) \in \SS D$.

The remaining part of the proof is quite similar
to the proof of the second part of theorem~\ref{theorem3} and omitted here.
 \end{proof}


\section{Integral angular moments of generalized ART}

Let a function $f(x,\xi) \in \SS D$, $p \geq
0$ being an integer, and ART $u_k(x,\xi)$ of order $k$ (\ref{TF_k-order})
be given. We define a symmetric $p$-tensor field ${\cal E}_{kp}$ 
by
 \begin{equation}\label{Ang_p-moment}
  {\cal E}_{kp}^{i_1\dots i_p}u_k (x)
  = \int\limits_{\SS^2_x} u_k(x,\xi)
  \xi^{i_1} \dots \xi^{i_p}\,
  d \lambda_x(\xi), \quad p \geq 0.
 \end{equation}
Here, $\xi^{j}$ is the $j$-component of unit vector $\xi$,
$j=1,2,3;\,$ $d\lambda_x(\xi)$ is the angular measure
on the sphere $\SS^2_x$. In particular, for $p=0$, we have
 \begin{equation}\label{Ang_0-moment}
  {\cal E}_{k0}u_k (x)
  = \int\limits_{\SS^2_x}u_k(x,\xi)\,d\lambda_x(\xi).
 \end{equation}

The tensor field ${\cal E}_{kp}u_k$
defined by (\ref{Ang_p-moment}),
is called the {\it integral angular $p$-moments} of ART
of order $k$. For brevity we call them {\it angular moments}
below. The field ${\cal E}_{k0}u_k$ is scalar, and the
fields ${\cal E}_{kp}u_k$, for an integer $p \geq 1$, are
symmetric tensor fields of rank $p$. Subsequently we use the
notations ${\cal E}_{kp}$ for the tensor fields ${\cal
E}_{kp} u_k$.

In this section we establish certain differential
relations connecting the tensor fields ${\cal E}_{kp}$
between themselves and with function $f(x,\xi)$. The
operator of divergence, acting on symmetric tensor
fields, is defined by
\begin{equation}\label{divg}
 \displaystyle{
   \left( \delta\,w \right)_{j_{1}\ldots j_{m}}
   = \frac{\partial w_{j_{1} \ldots j_{m}p}}
   {\partial x^q}\,\delta^{pq},
 }
\end{equation}
where $w_{j_{1}\ldots j_{m}}$ is a tensor field $w$ of rank
$m$, $\,\delta^{jk}$ denotes the Kronecker symbol.

We 
compute $\delta\,{\cal E}_{kp}$ for $\,k \geq 1$, $\,p \geq 1$.
By definition we have
 \begin{equation}\label{div_Ang_moment}
  (\delta\, {\cal E}_{kp})^{i_1\dots i_{p-1}}(x)
  = \frac {\partial }{\partial x^j}\,
  {\cal E}_{kp}^{i_1
  \dots i_{p-1} j}(x) = \int\limits_{\SS^2_x}
  \frac{\partial u_k(x,\xi)}{\partial x^j}
  \xi^{i_1}
  \dots \xi^{i_{p-1}} \xi^j\, d\lambda_x(\xi).
 \end{equation}
By means of lemma \ref{lemma1} and (\ref{H_coord}) we
obtain
$$
  \xi^j \frac{\partial u_k}{\partial x^j}
  = {\mathcal H} u_k
  = k\,u_{k-1}-\alpha\,u_k,
$$
where $u_k$ is defined by (\ref{TF_k-order}),
and the operator ${\mathcal H}$ is defined by
(\ref{H_def}). Substituting the last formula into
(\ref{div_Ang_moment})
we obtain
 \begin{equation}\label{Res_div_Ang_mom}
  (\delta\, {\cal E}_{kp})^{j_1\dots j_{p-1}}(x)
  = k\,{\cal E}^{j_1 \dots j_{p-1}}_{(k-1)(p-1)}(x) 
  - \displaystyle{ \int\limits_{\SS^2_x} \alpha(x,\xi)
  u_k(x,\xi) \xi^{i_1}\dots \xi^{i_{p-1}}
  d\lambda_x(\xi). }
 \end{equation}
If the function $\alpha(x,\xi) = \alpha (x)$ does not
depend on vector $\xi$, the expression
(\ref{Res_div_Ang_mom})
can be written as
 \begin{equation}\label{Res_div_Ang_mom-1pc}
  \delta\, {\cal E}_{kp} = k\,{\cal E}_{(k-1)(p-1)}
  - \alpha\,{\cal E}_{k(p-1)},
  \quad k \geq 1,\quad  p \geq 1.
 \end{equation}
If the function $\alpha(x,\xi)$ is a homogeneous
polynomial with respect to the variables $\xi^j$ (i.e. $\alpha (x,\xi) =
Q^r_{j_1 \dots j_r}(x) \xi^{j_1} \dots \xi^{j_r}$, where
$Q^r$ is $r$-tensor field, $r \geq 1$), then 
(\ref{Res_div_Ang_mom})
has the form
 \begin{equation}\label{Res_div_Ang_mom-2pc}
  \delta\, {\cal E}_{kp}
  = k\,{\cal E}_{(k-1)(p-1)}
  - Q^r\ast{\cal E}_{k(p+r-1)},
 \end{equation}
where $Q^r \ast {\cal E}_{k(p+r-1)}$ is a convolution of
tensor fields $Q^r$ and ${\cal E}_{k(p+r-1)}$,
$$
  \left( Q^r \ast {\cal E}_{k(p+r-1)}
  \right)^{j_1 \dots j_{p-1}}
  = Q^r_{l_1 \dots l_r}{\cal E}^{l_1 \dots l_r j_1
  \dots j_{p-1}}_{k(p+r-1)}.
$$

For $k=0$ the relation (\ref{div_Ang_moment}) has a different
form. Indeed, applying the operator $\delta$ to
the field ${\cal E}_{0p}$ and using lemma \ref{lemma2},
we obtain $\big( {\mathcal H} + \alpha \big) u_0 = f $,
and then
 \begin{equation}\label{Res_div_Ang_mom0}
  (\delta\, {\cal E}_{0p})^{j_1 \dots j_{p-1}}(x)
  = f^{j_1 \dots j_{p-1}}_{p-1}(x)
  - \int\limits _{\SS^2_x} \alpha(x,\xi)\,
  u_0(x,\xi) \xi^{j_1} \dots \xi^{j_{p-1}}\,
  d \lambda_x(\xi),
 \end{equation}
where $f_{p-1}(x)$ means the angular $(p-1)$-moment of the
function $f$. The components of $f_{p-1}(x)$ are
 \begin{equation}\label{p-1_Ang_mom_mu}
  f_{p-1}^{j_1 \dots j_{p-1}}(x)
  = \int\limits _{\SS^2_x}
  f(x,\xi)\xi^{j_1} \dots
  \xi^{j_{p-1}}\,d \lambda_{x}(\xi).
 \end{equation}
For $\alpha(x,\xi) \equiv \alpha (x)$ the formula
(\ref{Res_div_Ang_mom0}) 
has the form
 \begin{equation}\label{div_p0_Ang_mom}
  \delta\, {\cal E}_{0p}
  = f_{p-1} - \alpha\, {\cal E}_{0(p-1)},
 \end{equation}
and for $\alpha(x,\xi) = Q^r_{j_1 \dots j_r}(x)\xi^{j_1}
\dots \xi^{j_r}$ it has the form
 \begin{equation}\label{div_p0_Ang_mom_pc2}
  \delta\, {\cal E}_{0p}
  = f_{p-1} - Q^r \ast {\cal E}_{0(p+r-1)}.
 \end{equation}

\begin{remark} 
Analogous relations to (\ref{Res_div_Ang_mom-2pc}),
(\ref{div_p0_Ang_mom_pc2}) can be derived for
the sum of polynomials which are homogeneous with respect
to the variables $\xi^j$, j=1,2,3. The coefficients of each
homogeneous polynomial are components of a symmetric
tensor field. The tensor fields involved in different terms
of the sum may have different ranks.\\[1ex]
\end{remark}

The formulas
(\ref{Res_div_Ang_mom})-(\ref{div_p0_Ang_mom_pc2})
can be used for deriving additional relations. Let 
$\alpha(x,\xi) = \varepsilon + \i \kappa$ being
constant, i.e. $\varepsilon = {\rm const}$,
$\kappa = {\rm const}$. Applying the operator $\delta$ to
(\ref{Res_div_Ang_mom})
we get
$$
  \delta\, (\delta\, {\cal E}_{kp})
  = k\, \delta\, {\cal E}_{(k-1)(p-1)}
  - \alpha\, \delta\, {\cal E}_{k(p-1)},
$$
for $k \geq 2$, $p\geq 2$. Applying again
(\ref{Res_div_Ang_mom}) to
$\delta\, {\cal E}_{k(p-1)}$ and $\delta\, {\cal
E}_{(k-1)(p-1)}$, we obtain
$$
  \delta\,(\delta\, {\cal E}_{kp})
  = k(k-1)\,{\cal E}_{(k-2)(p-2)}
  - 2k\,\alpha\, {\cal E}_{(k-1)(p-2)}
  + \alpha^2{\cal E}_{k(p-2)}.
$$
A further application of (\ref{Res_div_Ang_mom})
yields equations for the operator $\delta^n$
over tensor fields ${\cal E}_{kp}$ with $k \geq n$, $p \geq
n$.

\medskip

 \begin{theorem}\label{theorem-1} Let $k$, $n$,
$p$ be integers, $n \geq 1$, $k \geq 0$, $p \geq 0$, and
${\cal E}_{(k+n)(p+n)}$ be a field of angular
$(p+n)$-moments of ART $u_{k+n}$ of order $(k+n)$,
$\alpha(x,\xi) = \varepsilon + \i \kappa$, $\varepsilon =
{\rm const}$, $\kappa ={\rm const }$. Then,
 \begin{equation}\label{div-n_Ang_mom}
  \delta^n {\cal E}_{(k+n)(p+n)}
  = \sum\limits^n_{j=0} (-1)^j \binom{n}{j} \alpha^j
  \frac{(k+n)!}{(k+j)!}{\cal E}_{(k+j)p},
 \end{equation}
where $\displaystyle \binom{n}{j} = \frac {n!}{j!(n-j)!}
$ is a binomial coefficient.
 \end{theorem}

 \begin{proof} We prove formula (\ref{div-n_Ang_mom})
by induction on $n$. For $n=1$ we have
(\ref{Res_div_Ang_mom}), 
with the integers $k$, $p$ changed to $k+1$, $p+1$. Assume
the equality (\ref{div-n_Ang_mom}) to be correct.
Applying the operator
$\delta$ and using (\ref{Res_div_Ang_mom}), 
we get 
$$
  \delta^{n+1}{\cal E}_{(k+n)(p+n)}
  = \displaystyle{ \sum\limits^n_{j=0} (-1)^j
  \binom{n}{j} \alpha^j
  \frac{(k+n)!}{(k+j)!} \big( (k+j)
  {\cal E}_{(k+j-1)(p-1)} } 
  \displaystyle{ - \alpha\, {\cal E}_{(k+j)(p-1)} \big). }
$$
Isolating the first and the last items on the right-hand side,
summing the rest in pairs and taking into account
that $\displaystyle \binom{n}{j}+\binom{n}{j-1} =
\binom{n+1}{j}$, we obtain
$$
 \begin{array}{rcl} 
  \delta^{n+1}{\cal E}_{(k+n)(p+n)}
& = & \displaystyle{ \sum\limits^n_{j=1} (-1)^j \alpha
  \frac{(k+n)!}{(k-1+j)!}\binom{n+1}{j}
  {\cal E}_{(k+j)(p-1)} } \vspace*{1mm} \\
& & {}+\, \displaystyle{ \frac{(k+n)!}{(k-1)!}
  {\cal E}_{(k-1)(p-1)}
  + (-1)^{n+1} \alpha^{n+1}{\cal E}_{(k+n)(p-1)}. }
 \end{array}
$$
Presenting the right-hand side of the obtained formula as a
sum with index $j$, $j=0, \ldots, n+1$,
$$
  \delta^{n+1}{\cal E}_{(k-1+n+1)(p-1+n+1)}
  = \sum\limits^{n+1}_{j=0} (-1)^j \alpha ^j
  \binom{n+1}{j} \frac{(k-1+n+1)!}{(k-1+j)!}
  {\cal E}_{(k-1+j)(p-1)},
$$
then changing the indexes $k-1$ and $p-1$ to $k$, $p$, we
derive an equation of a form (\ref{div-n_Ang_mom}),
with $n+1$ instead of $n$,
$$
  \delta^{n+1}{\cal E}_{(k+n+1)(p+n+1)}
  = \sum\limits^{n+1}_{j=0} (-1)^j \alpha^j
  \binom{n+1}{j} \frac{(k+n+1)!}{(k+j)!}
  {\cal E}_{(k+j)p}.
$$
This finishes the proof.
 \end{proof}
 
\medskip

A repeated application of operator $\delta$ to tensor
fields of type ${\cal E}_{0p}$ (i.e. at $k=0$) leads to
other equations. Using (\ref{div_p0_Ang_mom})
for $p \geq 2$, then the relation $\delta\,f_{p-1} =
({\mathcal H} f)_{p-2}$, where the field of angular
$(p-1)$-moments of a function $f$ is defined by
(\ref{p-1_Ang_mom_mu}),
we obtain
$$
  \delta^2{\cal E}_{0p} = ({\mathcal H} f)_{p-2}
  - \alpha\, f_{p-2} + \alpha^2{\cal E}_{0(p-2)}.
$$
A further application of the operator $\delta$ to tensor
fields ${\cal E}_{0(p+n)}$, $n \geq 1$, $p \geq 0$, leads
to a result that can be checked immediately.

\medskip

 \begin{corollary}\label{corol-1} Let $n$, $p$ be
integers, $n \geq 1$, $p \geq 0$, and ${\cal E}_{0(p+n)}$
be a field of angular $(p+n)$-moments of ART $u_{0}$ of
order $0$, $\alpha(x,\xi) = \varepsilon + \i k$,
$\varepsilon = {\rm const}$, $k ={\rm const }$. Then
 \begin{equation}\label{div-n_Ang_mom0}
  \delta^n {\cal E}_{0(p+n)}
  = \sum\limits^n_{j=0} (-1)^j \binom{n}{j} \alpha^j
  ({\mathcal H}^{n-j-1} f)_p + (-1)^n \alpha^n
  {\cal E}_{0p},
 \end{equation}
where ${\mathcal H}^{n-j-1}$ is a power of ${\mathcal
H}$, $({\mathcal H}^{n-j-1} f)_p$ is the angular
$p$-moment of the function ${\mathcal H}^{n-j-1} f$,
$({\mathcal H}^0 f)_p = f_p$.
 \end{corollary}
 
\medskip

The equations (\ref{div-n_Ang_mom}),
(\ref{div-n_Ang_mom0})
point out that tensor fields ${\cal E}_{kp}$, $p \geq 0$,
$k \geq 0$, are expressed by a direct formula using
iterated divergence $\delta^n$ of the fields ${\cal
E}_{nq}$ with $p\geq q$, with usage of
angular moments of ART with $k\geq n$, and
angular moments of function ${\mathcal H}^{k-j}f$. We
note that angular moments of function ${\mathcal
H}^{k-j}f$ can be treated physically as multipole
sources.

We consider now non-stationary generalized ART. Let
in $R \times \SS D$ a field $u_k(t,x,\xi)$ of order $k$
(\ref{TF_k-order-t})
be given.
Generalizations of equations (\ref{Ang_p-moment}),
(\ref{Ang_0-moment}),
depending on $t$ and $x$
$$
  {\cal E}_{kp}(t,x), \quad
  {\cal E}_{k0}(t,x)
$$
of angular $p$-moments of ART of order $m$, are
defined naturally for non-stationary sources. Using
the operators $\delta$, ${\mathcal H}$ and lemma
\ref{lemma3}, 
 $$
   \Big( \frac{\partial}{\partial t} + {\mathcal H}
   + \alpha \Big)u_k = k\, u_{k-1},
   \quad k \geq 1, \quad
   \Big( \frac{\partial}{\partial t} + {\mathcal H}
   + \alpha \Big) u_0 = f,
 $$
it is not difficult to get formulas of the type
(\ref{Res_div_Ang_mom})--(\ref{div-n_Ang_mom0})
for the non-stationary case. We formulate the non-stationary
variants of (\ref{div-n_Ang_mom}),
(\ref{div-n_Ang_mom0})
as a theorem.

\medskip

 \begin{theorem}\label{theorem1-t} Let $k$, $n$,
$p$ be integers, $n \geq 1$, $k \geq 0$, $p \geq 0$, and
${\cal E}_{(k+n)(p+n)}$ be a field of angular
$(p+n)$-moments of ART $u_{k+n}(t,x,\xi)$ of order
$(k+n)$ \eqref{TF_k-order-t}, $\alpha(x,\xi) =
\varepsilon + \i \kappa$, $\varepsilon = {\rm const}$,
$\kappa = {\rm const }$. Then,
 \vspace*{-2mm}
 $$
  \delta^n{\cal E}_{(k+n)(p+n)}
  = \sum\limits^n_{j=0}(-1)^j \binom{n}{j}
  \frac{(k+n)!}{(k+j)!} \Big(
  \frac{\partial}{\partial t} + \alpha \Big)^j
  {\cal E}_{(k+j)p}, \quad k \geq 1,
 $$
 \vspace*{-3mm}
 $$
  \delta^n{\cal E}_{0(p+n)}
  = \sum\limits^{n-1}_{j=0}(-1)^j
  \Big( \frac{\partial }{\partial t} + \alpha \Big)^j
  \big( {\mathcal H}^{n-j-1}f \big)_p + (-1)^n
  \Big( \frac{\partial }{\partial t}
  + \alpha \Big)^n {\cal E}_{0p} 
 $$
Here, ${\mathcal H}^{n-j-1}$ is a power of the operator
${\mathcal H}$, $({\mathcal H}^{n-j-1} f)_p$ is an
angular $p$-moment of a function ${\mathcal H}^{n-j-1}
f$, $({\mathcal H}^0 f)_p = f_p$.
 \end{theorem}
 
\medskip

Consider the partial case of dependance of the source
distribution $f(x,\xi)$ on points $x \in D$ only, i.e.
$f(x,\xi) \equiv f(x)$.

We aim for finding scalar fields ${\mathcal{E}}_{k0}$ of angular
$0$-moments over ART of order $k$, $k \geq 1$, determined
by the formula (\ref{Ang_0-moment}). Substituting in
(\ref{Ang_0-moment}) the definition (\ref{TF_k-order}) of
ART of order $k$ and taking into account that $s=|x-q|$,
$\displaystyle \xi = \frac {x-q}{|x-q|}$ and
$\displaystyle{ dV_q = s^2 d s\,d\lambda_x(\xi) }$, we
obtain at $\alpha (x,\xi) = {\rm i} \kappa$, $\kappa =
{\rm const}$,
\begin{align*}
 {\mathcal E}_{k0}(x)
  =  \int\limits_{\SS^2_x}
  \hspace*{-1mm} u_k(x,\xi)\,
  d \lambda_x(\xi)   
&  =  \int\limits_{\SS^2_x}
  \int\limits_0^{\infty} \hspace*{-1mm} s^k \exp
  \{ {\rm i} \kappa s \}\,
  f(x-s\xi)\,d s\,d\lambda_x(\xi)\\
&  = 
  \int\limits_{R^3} 
  \frac{ \exp \{ {\rm i} \kappa |x-q|\}}
  {|x-q|^{2-k}} f(q)\, dV_q
\end{align*}
for $\, k=1,2, \dots$. The right-hand side of the last
formula is, for $k=1$, a volume potential satisfying the
Helmholtz equation
$$
  \Delta {\mathcal E}_{10}
  + \kappa^2{\mathcal E}_{10} = -4\pi f
$$
and Sommerfeld's radiation condition.

We denote by $G_k(r)$ the functions $r^{k-2}
e^{-\alpha r}$ for $k=1,2,\ldots$, where $r=|x-q|$,
$\alpha = \varepsilon + {\rm i} \kappa$, and suppose that the
constants $\varepsilon \geq 0$ and $\kappa$ do not vanish
simultaneously. An application of the operator $(\Delta -
\alpha^2)$ on these functions coincides with the radial part
of the Laplace operator since $G_k$ depends on $r$,
 \begin{equation}\label{Lap_op-Rad_part}
   \Big( \frac{\partial^2}{\partial r^2}
   + \frac{2}{r} \frac{\partial}{\partial r} \Big)
   G_k(r) = (k-1)(k-2)G_{k-2} - 2\alpha (k-1) G_{k-1}
   + \alpha^2 G_k.
 \end{equation}
It is easy to see that the operator $(\Delta - \alpha^2)$
connects the functions $G_{k}$ with different indexes $k$
with each other. Here are concrete formulas for
$G_{k}(r)$ for $k=2,3,4,5$,
 $$
  \begin{array}{ll}
   (\Delta - \alpha^2) G_2 = -2 \alpha G_{1},
   \vspace*{2mm}
   & (\Delta - \alpha^2) G_3 = 2 G_1 - 4 \alpha G_2,\\
   (\Delta - \alpha^2) G_4 = 6 G_2 - 6 \alpha G_3,
   & (\Delta - \alpha^2) G_5 = 12G_3 - 8 \alpha G_4.
  \end{array}
 $$
The relations (\ref{Lap_op-Rad_part}) allow to establish
classes of solutions for homogeneous equations with
operators $(\Delta - \alpha^2)^k$ with their
fundamental solutions. In particular if $r \neq 0$ then
$G_j(r)$, $j=1,\ldots,k$, are the solutions of homogeneous
equation with the operator $(\Delta - \alpha^2)^k$, $k$
is positive integer.

We derive a differential equation with partial
derivatives whose solution coincides with the angular
moment ${\mathcal E}_{20}$.

\medskip

 \begin{proposition}\label{propos1}
Suppose that $f(x)$ is a finite and infinitely differentiable function in
$\RR^3$, $\alpha(x,\xi) = \varepsilon + {\rm i}
\kappa $ is constant, $\varepsilon \geq 0$, $\alpha \neq
0$. If the angular moment ${\mathcal E}_{20}$ is defined
by \eqref{Ang_p-moment} for
$k=2$, then the function $\displaystyle{ u(x) = \frac{
{\mathcal E}_{20}(x)}{8\pi\alpha} }$ is a solution of the
equation
$$
  (\Delta - \alpha ^2)^2u = f.
$$
 \end{proposition}

 \begin{proof}
Applying on both parts of the formula
$$
  {\mathcal E}_{20}(x)
  = \int_{R^3} e^{-\alpha |x-q|}
  f(q)\,dV_q,
$$
the operator $(\Delta_x - \alpha^2)$, we get
 \begin{equation}\label{Bi_Helm}
  (\Delta_x - \alpha^2) {\mathcal E}_{20}(x)
  = \int_{R^3}f(q)
  \big( ( \Delta_x - \alpha^2 )
  e^{-\alpha |x-q|} \big) dV_q
 \end{equation}
and, after simple calculations,
$$
  (\Delta_x - \alpha^2)e^{-\alpha |x-q|}
  = -2\alpha \frac{e^{-\alpha |x-q|}}{|x-q|}.
$$
Representing
(\ref{Bi_Helm}) 
using the last expression,
$$
  (\Delta _x-\alpha ^2) {\mathcal E}_{20}(x)
  = \int_{R^3}(-2\alpha)
  \frac{e^{-\alpha |x-q|}}{|x-q|}f(q)\,dV_q,
$$
and applying to it the operator $\Delta_x - \alpha^2$
once more, we obtain
$$
  (\Delta_x - \alpha^2)^2 {\mathcal E}_{20}(x)
  = -2\alpha \int_{R^3}
  (\Delta _x-\alpha ^2) \frac{e^{-\alpha |x-q|}}{|x-q|}
  f(q)\,dV_q = 8 \pi \alpha f(x),
$$
which proves the statement.
 \end{proof}


\section{Conclusion}

In this article we conisdered the generalized attenuated ray transforms
(ART) for
functions $f(x,\xi)$ and symmetric tensor fields
$w(x,\xi)$ defined on the spherical bundle (phase space).
The investigated transforms are connected with attenuated ray
transform arising in computerized, emission and tensor
tomography problems, wave optics and photometry and
integral geometry.
The generalizations were performed in four directions. At
first, a source distribution $f$ and an absorption
coefficient
$\varepsilon$ depended on the vector of
direction $\xi$. Secondly, the coefficient $\varepsilon$
was complex-valued, and, third, the weights of
generalized ART had a more general form and contained
monomials. The last generalization was that our
mathematical model contained internal sources $f(t,x,\xi)$
(in scalar case) or symmetric tensor fields $w(t,x,\xi)$
that depended on time.

The generalization of ART operator led to stationary
ART $u_k(x,\xi)$ and non-stationary ART $u_k(t,x,\xi)$ of
order $k$ over functions $f$ and, in particular, over
$m$-tensor fields $w$. They could be treated as the
integral moments of a source distribution $f$ or of a
symmetric tensor field $w$ with components $w_{i_1 \ldots
i_m}$ with a weight generated by an exponential function.
Connections between ART of different orders are
established. Differential equations whose solutions are
the generalized ART-operators of order $k$
were derived. In particular,
the differential equations of the first order coincided
with stationary and non-stationary transport equations
with complex-valued absorption coefficient, but without
integral part that is responsible for the scattering phenomenon
\cite{KeizZwa72}. Uniqueness theorems for boundary-value
problems in stationary case, and initial boundary-value
problems in non-stationary case have been proved.

The back-projection operator (BPO) is an important
instrument regarding the inversion of tomographic
operators within the computer, emission and tensor
tomography. Tensor fields of the integral angular
$p$-moments are one of possible generalizations of BPO.
As well as the back-projection operators, angular moments
do not allow to find sought-for object by its tomographic
data, but they show certain characteristic features of
the object. Besides that they can be treated as 
conservation laws or applied as special additional
(a-priori) information for the development of inversion
algorithms. In particular, if $f(x,\xi)$ has the form
(\ref{Part-TFk}), the angular moment (\ref{Ang_p-moment})
coincides with the back-projection operator over the longitudinal
ray transform. Properties of the operators of angular
moments of order $p$ were investigated and connections
between the moments of different orders were detected.

There exist close connections of ART of order $k$
with different problems of optics, tomography and
integral geometry.
According to optical terminology it can be seen easily
that $u_1(x,\xi)$ for $\varepsilon \equiv 0$ and $\rho
(x,\xi) = Im \alpha = \mbox{const}$ is the ideal wave
image, and $u_0(x,\xi)$ for $\rho \equiv 0$ and
$\varepsilon = \mbox{const} $ is the ideal photometric
image \cite{Kirt75,BornWolf73,Kirt83}. In terms of computerized tomography the operator
(\ref{TF_k-order}) for
$\rho=0$, $\varepsilon=0$ can be seen as fan-beam or
cone-beam transform, and as well as the well-known Radon or
ray transform. In more complicated mathematical models,
for example in emission tomography, the operator
(\ref{TF_k-order}) is the standard attenuated ray transform,
and a certain natural generalization of the integrand leads
to the longitudinal ray transform of symmetric
tensor fields and to integral moments of generalized
tensor fields \cite{Shar94}.

The generalized ART that has been investigated in the article and angular
moments are connected in their partial cases with various
transforms of tomographic types and back-projection
operators. So it arises as natural settings of inverse
problems of determination of scalar, vector or tensor
fields by their known ART of order $k$. These inverse
problems can be treated as inverse problems which consists of determining its
the right-hand side of a
generalized transport equation. In that sense the concepts that
have been considered in the article have good potential
for further development and
investigation in these fields.
In particular, the generalized ART can be extended to the
case of an absorption coefficient $\varepsilon$ that depends on
time $t$. Besides that the extension of ART to the case of a
Riemannian metric including stationary and
non-stationary settings looks promising.
The mentioned generalizations might lead to the construction and
subsequent investigations of spacious mathematical models
for dynamic refractive tensor tomography.


\section*{Acknowledgements}

This work was supported by the
Programm for Basic Researches SB RAS No. I.1.5 (project
0314-2016-0011), Russian Foundation for Basic Research
(RFBR) and German Science Foundation (DFG) according to
the research project 19-51-12008, and German Science
Foundation (DFG) under Lo~310/17-1.


\newpage

\bibliographystyle{siam}
\bibliography{references_arXiv}

\end{document}